\documentclass{amsart}
\usepackage{graphicx}
\usepackage[all]{xy}
\usepackage{amscd}
\usepackage{amssymb}
\usepackage{amsmath}
\usepackage{amsthm}
\usepackage{amsfonts}
\usepackage{graphics}
\usepackage{epsfig,psfrag}
\usepackage[english]{babel}
\usepackage{latexsym}
\usepackage{mathrsfs}

\newtheorem{theorem}{Theorem}[section]
\newtheorem{corollary}[theorem]{Corollary}
\newtheorem{lemma}[theorem]{Lemma}
\newtheorem{claim}{Claim}
\newtheorem{proposition}[theorem]{Proposition}

\newtheorem{step}[theorem]{Step}
\newtheorem{question}[theorem]{Question}

\theoremstyle{definition}

\newtheorem{definition}[theorem]{Definition}

\theoremstyle{remark}
\newtheorem{remark}[theorem]{Remark}

\begin{document}

\title[A classification of maximally symmetric surfaces in $T^3$]
{A classification of maximally symmetric surfaces in the 3-dimensional torus}

\author{Chao Wang}
\address{Jonsvannsveien 87B, H0201, Trondheim 7050, NORWAY}
\email{chao\_{}wang\_{}1987@126.com}

\author{Bruno Zimmermann}
\address{Dipartimento di Matematica e Geoscienze, Universita degli Studi di Trieste, Trieste 34100, ITALY}
\email{zimmer@units.it}



\thanks{The first author was supported by NNSFC (No. 11501534).}

\begin{abstract}
If a finite group of orientation-preserving diffeomorphisms of the 3-dimensional torus leaves invariant an oriented, closed, embedded surface of genus $g>1$ and preserves the orientation of the surface, then its order is bounded from above by $12(g-1)$. In the present paper we classify (up to conjugation) all such group actions and surfaces for which the maximal possible order $12(g-1)$ is achieved, and note that the unknotted surfaces can be realized by equivariant minimal surfaces in a 3-torus.
\end{abstract}

\date{}
\maketitle

\section{Introduction}\label{sec:introduction}
All manifolds and maps considered in the present paper are smooth, and group actions are faithful and orientation-preserving.

\begin{definition}\label{def:extendable}
Let $\Sigma_g$ be a closed, connected, orientable surface of genus $g>1$ and $G$ be a finite group. A $G$-action on $\Sigma_g$ is {\it extendable} over a 3-manifold $M$ with respect to an embedding $e:\Sigma_g\rightarrow M$ if there is a $G$-action on $M$ such that $h\circ e=e\circ h$, for all $h\in G$. Identifying $\Sigma_g$ with $e(\Sigma_g)$, we will also shortly say that $G$ acts on the pair $(M,\Sigma_g)$. We will always assume $g>1$ in the present paper.
\end{definition}

A classical result of Hurwitz says that the order of a finite group action on a surface $\Sigma_g$ is bounded by $84(g-1)$ \cite{Hu}, and an action realizing this bound is usually called a Hurwitz action; in general, the Hurwitz actions are not classified. For actions on surfaces which extend to a 3-dimensional handlebody, the upper bound is $12(g-1)$ \cite{Z1}, and again the actions realizing this upper bound are not classified. More generally, if we require that the actions on surfaces extend to a certain 3-manifold $M$, then there will be an upper bound and hopefully one can classify the actions realizing this bound.

The most natural choices of $M$ include the 3-dimensional Euclidean space $\mathbb{R}^3$ and the 3-dimensional sphere $\mathbb{S}^3$. In each case, the classification does exist, and it is stronger in the sense that for each given $g$ the maximum of the group order can be obtained and the actions realizing the maximum can be classified (see \cite{WWZZ3} for $\mathbb{R}^3$ and \cite{WWZZ1,WWZZ2} for $\mathbb{S}^3$).

The 3-dimensional torus $\mathbb{T}^3$ is another natural choice. In this case, the upper bound is again $12(g-1)$ as shown in \cite{BRWW}, as a consequence of the equivariant loop theorem \cite{MY} (since an embedded surface of genus $g>1$ in the 3-torus has to be compressible) and the formula of Riemann-Hurwitz. In \cite{BRWW}, various series of actions of maximal possible order $12(g-1)$ are constructed and a conjectural picture of the situation is given. In the present paper, we obtain a complete classification for the maximal case $12(g-1)$; in particular, this confirms the conjecture in \cite{BRWW}.

\begin{definition}\label{def:knotted}
A closed subsurface in a closed 3-manifold is {\it unknotted} if the surface separates the 3-manifold into two handlebodies (so it is a Heegaard surface of a Heegaard splitting of the 3-manifold), otherwise the surface is {\it knotted}.
\end{definition}

\begin{theorem}\label{thm:classification}
If a $G$-action on $\Sigma_g$ with order $12(g-1)$ is extendable over $\mathbb{T}^3$, then $g-1$ has one of the forms: $2n^3$, $4n^3$, $8n^3$, $n^2$, $3n^2$ where $n\in\mathbb{Z}_+$ is a positive integer. Up to conjugation, such actions on the pair $(\mathbb{T}^3,\Sigma_g)$ are listed below:
\[\begin{array}{ccc|ccccc|c}
  2n^3 & 2n^3 & 2n^3 & 2n^3(2\nmid n) & 2n^3(2\nmid n) & 2n^3(2\nmid n) & 2n^3(2\nmid n) & 2n^3(3\nmid n) & n^2\\
  4n^3 & 4(2n)^3 & 4n^3 & & & 4n^3(2\nmid n) & & 4n^3(3\nmid n) & 3n^2\\
  8n^3 & 8n^3 & 8n^3 & & 8n^3(2\nmid n) & & & 8n^3(3\nmid n) &
\end{array}\]
Here each number $m$ represents an action on $(\mathbb{T}^3,\Sigma_{m+1})$. If $m$ appears $k$ times, then there are $k$ different actions. The actions in the first three columns correspond to unknotted surfaces, all others to knotted ones.
\end{theorem}

For example, since $64=8\times 2^3=8^2$, there are five actions on $(\mathbb{T}^3,\Sigma_{65})$ realizing the maximal order $12\times 64=768$. For three of them the surface is unknotted and for two of them the surface is knotted. We will derive Theorem \ref{thm:classification} from a stronger classification result Theorem \ref{thm:main} in Section \ref{sec:find coverings}, and the nine columns in Theorem \ref{thm:classification} correspond to the nine cases of Theorem \ref{thm:main}. Theorem \ref{thm:main} shows that all actions realizing the bound $12(g-1)$ are actually listed in the examples of \cite{BRWW}, in particular it follows that all the unknotted surfaces can be realized by equivariant minimal surfaces.

\begin{corollary}\label{cor:minimal}
If a Heegaard surface $\Sigma_g$ of $\mathbb{T}^3$ is invariant under a finite group action of order $12(g-1)$, then it can be realized by an equivariant minimal surface for some Euclidean structure on $\mathbb{T}^3$.
\end{corollary}

This confirms the following natural question for Euclidean 3-manifolds. Actually, the 3-torus and the Hantzsche-Wendt manifold (see \cite{Z2}) are the only orientable closed Euclidean 3-manifolds containing such surfaces. The question is also partly confirmed for spherical 3-manifolds. By \cite{La}, \cite{KPS} and \cite{BWW}, it is true for the case of $\mathbb{S}^3$ with three possible exceptions.

\begin{question}
If a Heegaard surface $\Sigma_g$ of an orientable closed geometric 3-manifold $M$ is invariant under a finite group action of order $12(g-1)$, can it be realized by an equivariant minimal surface for some geometric structure on $M$?
\end{question}

An example of a hyperbolic 3-manifold with such a Heegaard surface is the Seifert-Weber dodecahedral space, obtained by identifying opposite faces of a regular hyperbolic dodecahedron with dihedral angles $2\pi/5$, after a twist by $3(2\pi/10)$ of each face (\cite{SW}, [Th2,p.36]). After the identifications, the boundary of a regular neighborhood of the 12 edges connecting the center of the dodecahedron with the centers of its 12 faces gives a Heegaard surface of genus 6, and by [Po,Figure 2(d)], this Heegaard surface can be realized by an equivariant minimal surface, invariant under the action of isometry group $A_5$ of the dodecahedron. Applying the same construction to the regular spherical dodecahedron with dihedral angles $2\pi/3$ and twisting by $2\pi/10$, one obtains the spherical Poincar\'e sphere, with a Heegaard surface of genus 6 invariant under the dodecahedral group $A_5$, and by \cite{KPS} this can again be realized by an equivariant minimal surface. Finally, applying the construction to the Euclidean cube instead, one obtains the 3-torus with a Heegaard surface of genus 3 which can be realized by minimal surface invariant under the
isometry group $S_4$ of the cube (corresponding to the case $n=1$ of the first item in the first row of Theorem \ref{thm:classification}).

To classify the actions in Theorem \ref{thm:classification}, we need the orbifold theory (see \cite{BMP,Du,Th}). After identifying $\Sigma_g$ with $e(\Sigma_g)$, an extendable action gives an orbifold pair $(M/G,\Sigma_g/G)$. Conversely, given a 2-orbifold $\mathcal{F}$ in a 3-orbifold $\mathcal{O}$ and a regular orbifold covering $p:M\rightarrow\mathcal{O}$, if $p^{-1}(\mathcal{F})$ is connected, then the group $\pi_1(\mathcal{O})/\pi_1(M)$ acts on the pair $(M,p^{-1}(\mathcal{F}))$. If $M$ is $\mathbb{R}^3$ or $\mathbb{S}^3$, then finding all the pairs $(\mathcal{O},\mathcal{F})$ is enough (as in \cite{WWZZ2,WWZZ3}), because $\mathbb{R}^3$ and $\mathbb{S}^3$ are simply connected and $p$ is determined by $\mathcal{O}$. For $\mathbb{T}^3$ further information about the covering $p$ is needed.

Let $G$ be a finite group which acts on a pair $(\mathbb{T}^3,\Sigma_g)$ and has order $12(g-1)$; by the Riemann-Hurwitz formula, the quotient 2-orbifold $\mathcal{F}=\Sigma_g/G$ is a sphere with four singular points of indices $2,2,2,3$; by the geometrization of finite group actions on 3-manifolds, we can assume that $G$ acts by Euclidean isometries on $\mathbb{T}^3$, for some Euclidean structure on $\mathbb{T}^3$. Then the quotient 3-orbifold $\mathcal{O}=\mathbb{T}^3/G$ is a Euclidean orbifold; by a Bieberbach theorem, there is a minimal covering $p_0:\mathbb{T}^3\rightarrow\mathcal{O}$ such that for any covering $p:\mathbb{T}^3\rightarrow\mathcal{O}$ there is a covering $q:\mathbb{T}^3\rightarrow\mathbb{T}^3$ satisfying $p=p_0\circ q$. Hence the classification factors into two steps:

\begin{step}\label{ste:list pair}
List all pairs $(\mathcal{O},\mathcal{F})$ such that $\mathcal{O}$ is a Euclidean 3-orbifold and $\mathcal{F}$ is a sphere with four singular points of indices $2,2,2,3$.
\end{step}

\begin{step}\label{ste:find covering}
For a given pair $(\mathcal{O},\mathcal{F})$ in Step \ref{ste:list pair} find all coverings $q$ such that $p=p_0\circ q$ is a regular covering and $p^{-1}(\mathcal{F})$ is connected.
\end{step}

In section \ref{sec:list pairs}, we will finish Step \ref{ste:list pair} by using Dunbar's list of Euclidean 3-orbifolds. In section \ref{sec:find coverings}, we will finish Step \ref{ste:find covering} by finding all the possible normal subgroups of $\pi_1(\mathcal{O})$ corresponding to $p$. In section \ref{sec:example}, we will give an explicit example.

\section{List the pairs $(\mathcal{O},\mathcal{F})$}\label{sec:list pairs}
The way to list the orbifold pairs in Step \ref{ste:list pair} is similar to \cite{WWZZ2}. First, we need some general results and conventions from \cite{Du}.

In \cite{Du}, the Euclidean 3-orbifolds are classified. There are two classes: the fibred ones and the non-fibred ones. The fibred ones are the Seifert fibred orbifolds having Euler number $0$ and base 2-orbifold with Euler characteristic $0$. The non-fibred ones are listed in \cite{Du}. Moreover, the singular sets, which are trivalent graphs, of the Euclidean orbifolds with underlying space $\mathbb{S}^3$ are pictured, and the names of the fundamental groups of the orbifolds are given.

\begin{definition}\label{def:2-orbifold}
Let $\mathbb{S}^2(d_1,\ldots,d_s)$ denote the 2-orbifold which is a sphere with $s$ singular points of indices $d_1,\ldots,d_s$; let $\mathbb{D}^2(d_1,\ldots,d_s)$ denote the 2-orbifold which is a disk with $s$ singular points of indices $d_1,\ldots,d_s$ in the interior.

Let $\mathbb{D}^2(d_1,\ldots,d_s;n_1,\ldots,n_r)$ be the 2-orbifold which is a disk with $s$ singular points of indices $d_1,\ldots,d_s$ in the interior and $r$ corner points of groups $D_{n_1},\ldots,D_{n_r}$ in the boundary. Here $D_n$ denotes the dihedral group of order $2n$, and the boundary points other than the corner points are reflection points.
\end{definition}

\begin{lemma}\label{lem:fibred}
If $(\mathcal{O},\mathcal{F})$ is a pair as in Step \ref{ste:list pair} and $\mathcal{O}$ is fibred, then the base 2-orbifold of $\mathcal{O}$ is $\mathbb{D}^2(-;2,3,6)$. As a consequence, $\mathcal{O}$ has underlying space $\mathbb{S}^3$.
\end{lemma}

\begin{proof}
Let $B$ be the base 2-orbifold of $\mathcal{O}$. Then the Euler characteristic of $B$ is $0$. Since $\mathcal{F}$ has singular points with indices $2$ and $3$, $B$ is one of $\mathbb{S}^2(2,3,6)$, $\mathbb{D}^2(3;3)$, $\mathbb{D}^2(-;3,3,3)$ and $\mathbb{D}^2(-;2,3,6)$. If $B$ is $\mathbb{S}^2(2,3,6)$, since the Euler number of $\mathcal{O}$ is $0$, then $\mathcal{O}$ can only be $\mathbb{S}^2(2,3,6)\times\mathbb{S}^1$ which has no suborbifolds of type $\mathbb{S}^2(2,2,2,3)$. If $B$ is $\mathbb{D}^2(3;3)$ or $\mathbb{D}^2(-;3,3,3)$, then by the discussion in section 4 and 5 of \cite{Du} the underlying space of $\mathcal{O}$ is a lens space or $\mathbb{S}^3$. Then $\mathcal{F}$ separates $\mathcal{O}$. Since the singular set of index $2$ in $\mathcal{O}$ consists of circles (with degree $3$ singular points removed), it cannot intersect $\mathcal{F}$ three times. Hence $B$ is $\mathbb{D}^2(-;2,3,6)$ and by the discussion in section 4 and 5 of \cite{Du} the underlying space of $\mathcal{O}$ is $\mathbb{S}^3$.
\end{proof}

\begin{lemma}\label{lem:nonfibred}
If $(\mathcal{O},\mathcal{F})$ is a pair as in Step \ref{ste:list pair} and $\mathcal{O}$ is non-fibred, then $\mathcal{O}$ has underlying space $\mathbb{S}^3$.
\end{lemma}

\begin{proof}
By the classification result in \cite{Du}, the only non-fibred Euclidean 3-orbifold with underlying space not homeomorphic to $\mathbb{S}^3$ has underlying space $\mathbb{RP}^3$, and its singular set of index $3$ consists of a circle. Then $\mathcal{F}$ separates $\mathcal{O}$ and cannot intersect the circle only once.
\end{proof}

\begin{lemma}\label{lem:handlebody}
If $(\mathcal{O},\mathcal{F})$ is a pair as in Step \ref{ste:list pair}, then $\mathcal{F}$ bounds a handlebody orbifold which is a regular neighborhood of an edge of the singular set, with boundary $\mathbb{S}^2(2,2,2,3)$.
\end{lemma}

\begin{proof}
As in \cite{BRWW}, the equivariant loop theorem \cite{MY} gives a compression disk $\mathbb{D}^2(d)$ of $\mathcal{F}$. Since $\mathcal{F}$ is isomorphic to $\mathbb{S}^2(2,2,2,3)$, the compression disk splits $\mathcal{F}$ into two orbifolds $\mathbb{S}^2(2,2,d)$ and $\mathbb{S}^2(2,3,d)$. Then $d\geq 2$.

If $d\geq 7$, then $\mathbb{S}^2(2,3,d)$ has negative Euler characteristic and is incompressible, which contradicts the equivariant loop theorem. If $d\leq 5$, then both $\mathbb{S}^2(2,2,d)$ and $\mathbb{S}^2(2,3,d)$ are spherical and bound discal 3-orbifolds (as in \cite{WWZZ2}), because $\mathbb{T}^3$ is irreducible. Since each one of $\mathbb{S}^2(2,2,d)$ and $\mathbb{S}^2(2,3,d)$ cannot lie in the discal 3-orbifold bounded by the other one, the union of the two discal 3-orbifolds is the handlebody orbifold bounded by $\mathcal{F}$.

If $d=6$, then by the classification result in \cite{Du} the orbifold $\mathcal{O}$ is fibred. Then by Lemma \ref{lem:fibred} the base 2-orbifold of $\mathcal{O}$ is $\mathbb{D}^2(-;2,3,6)$. There is only one such $\mathcal{O}$ having singular points of index $6$. Its singular set is pictured as $[\text{P}622]$ in Figure \ref{fig:pairs}, where $\mathcal{F}$ can only be the boundary of a regular neighborhood of the edge $\beta$, up to isomorphism between the pairs $(\mathcal{O},\mathcal{F})$.
\end{proof}

\begin{proposition}\label{pro:pairs}
Up to isomorphism between orbifold pairs, all orbifold pairs $(\mathcal{O},\mathcal{F})$ in Step \ref{ste:list pair} are obtained as follows. The underlying topological space of the orbifold $\mathcal{O}$ is $\mathbb{S}^3$, and its singular set is given by one of the six pictures in Figure \ref{fig:pairs}. The 2-suborbifold $\mathcal{F}$, of type $\mathbb{S}^2(2,2,2,3)$, is obtained as the boundary of a regular neighborhood of one of the nine marked singular edges $\alpha$, $\beta$ or $\gamma$.

\begin{figure}[h]
\includegraphics{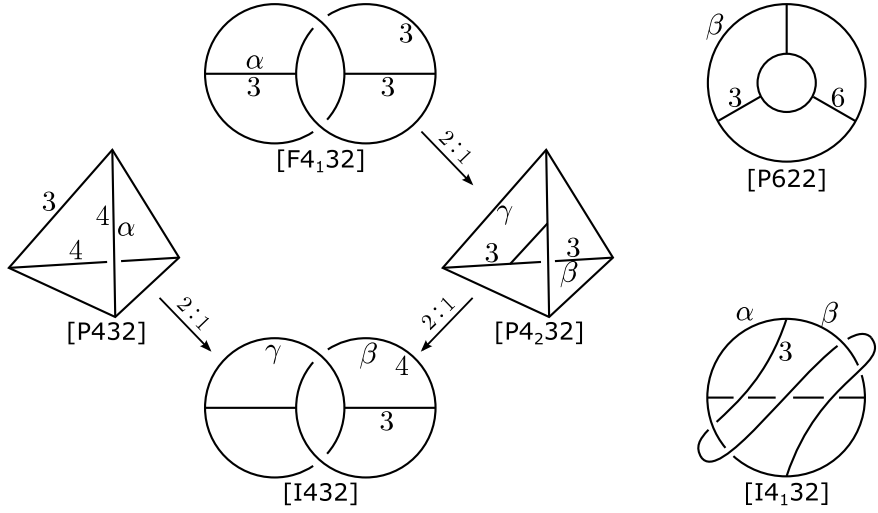}
\caption{Edges without number have index $2$. Names denote the corresponding space groups. Arrows indicate $2$-fold coverings.}\label{fig:pairs}
\end{figure}
\end{proposition}

\begin{proof}
By Lemma \ref{lem:fibred} and \ref{lem:nonfibred}, $\mathcal{O}$ has underlying space $\mathbb{S}^3$. Hence its singular set belongs to the list of pictures in \cite{Du}. By Lemma \ref{lem:handlebody}, $\mathcal{F}$ is the boundary of a regular neighborhood of a singular edge in the singular set of $\mathcal{O}$. Hence all the possible pairs $(\mathcal{O},\mathcal{F})$ can be found by enumerating the possible singular edges, which are exactly the marked edges in Figure \ref{fig:pairs}.
\end{proof}

\begin{remark}
Consider the complement of a regular neighborhood of a marked singular edge in Figure \ref{fig:pairs}. It is a handlebody orbifold if and only if the singular edge has mark $\alpha$. Hence only the singular edges $\alpha$ correspond to unknotted surfaces; the edges $\beta$ and $\gamma$ correspond to knotted ones.
\end{remark}

\section{Find the coverings $q$}\label{sec:find coverings}
For a given pair $(\mathcal{O},\mathcal{F})$ in Step \ref{ste:list pair}, we will first list the finite index normal translation subgroups of $\pi_1(\mathcal{O})$, then we will use a lemma in \cite{WWZZ2} to verify the connectedness.

For each $\mathcal{O}$ in Figure \ref{fig:pairs}, the representation of $\pi_1(\mathcal{O})$ as a space group can be found in \cite{Ha}. In the present paper we will use a slightly different representation of $\pi_1(\mathcal{O})$. First we need to introduce some notation (following \cite{BRWW}).

\begin{definition}\label{def:isometry}
Any element $t=(a, b, c)\in \mathbb{R}^3$ can act on $\mathbb{R}^3$ as the translation:
\[t: (x, y, z)\mapsto(x+a, y+b, z+c).\]
Let $t_x,t_y,t_z,t_{1/2},t_\omega$ be the following elements in $\mathbb{R}^3$ respectively:
\[(1,0,0),\quad (0,1,0),\quad (0,0,1),\quad (\frac{1}{2},\frac{1}{2},\frac{1}{2}),\quad (-\frac{1}{2},\frac{\sqrt{3}}{2},0).\]
For $n\in \mathbb{Z}_+$, let $T_{n^3/2},T_{n^3},T_{2n^3},T^\omega_{n^2},T^\omega_{3n^2}$ be the following subgroups of $\mathbb{R}^3$:
\begin{align*}
T_{n^3/2}&=\langle nt_x,nt_y,nt_{1/2}\rangle,\quad T_{n^3}=\langle nt_x,nt_y,nt_z\rangle,\\
T_{2n^3}&=\langle 2nt_x,nt_y+nt_x,nt_z+nt_x\rangle,\\
T^\omega_{n^2}&=\langle nt_\omega,nt_x,t_z\rangle,\quad T^\omega_{3n^2}=\langle 2nt_\omega+nt_x,nt_\omega+2nt_x,t_z\rangle.
\end{align*}
Let $r_y,r_z,r_{xy},r_{xyz},r_\omega$ be the following isometries of $\mathbb{R}^3$:
\begin{align*}
r_y: (x,y,z)&\mapsto(-x,y,-z),\quad r_z: (x,y,z)\mapsto(-x,-y,z),\\
r_{xy}: (x,y,z)&\mapsto(y,x,-z),\quad r_{xyz}: (x,y,z)\mapsto(z,x,y),\\
r_\omega: (x,y,z)&\mapsto(-\frac{1}{2}x-\frac{\sqrt{3}}{2}y,\frac{\sqrt{3}}{2}x-\frac{1}{2}y,z).
\end{align*}
\end{definition}

Note that when $m$ is one of $n^3/2$, $n^3$, $2n^3$, then $\mathbb{R}^3/T_m$ is homeomorphic to $\mathbb{T}^3$ with volume $m$, and when $m$ is one of $n^2$, $3n^2$, then $\mathbb{R}^3/T^\omega_m$ is homeomorphic to $\mathbb{T}^3$ with volume $\sqrt{3}m/2$. Moreover, we have
\[T_{n^3/2}\supset T_{n^3}\supset T_{4n^3},\quad T_{n^3}\supset T_{2n^3} \supset T_{(2n)^3},\quad T^\omega_{n^2}\supset T_{3n^2}\supset T_{(3n)^2}.\]
The isometries $r_y$, $r_z$ and $r_{xy}$ are $\pi$-rotations about the directions $(0,1,0)$, $(0,0,1)$ and $(1,1,0)$ respectively. The isometries $r_{xyz}$ and $r_\omega$ are right-hand $2\pi/3$-rotations about the directions $(1,1,1)$ and $(0,0,1)$ respectively.

\begin{lemma}\label{lem:rep.}
The universal covering groups of the 3-orbifolds in Proposition \ref{pro:pairs} are generated by the following elements, starting with the translation groups (whose indices in the whole groups are always $24$ except in the last case where the index is $12$):
\begin{itemize}
\item $[\text{P}432]$ : $\langle T_1,r_y,r_z,r_{xy},r_{xyz}\rangle$
\item $[\text{F}4_132]$ : $\langle T_2,r_y,r_z,t_{1/2}r_{xy},r_{xyz}\rangle$
\item $[\text{I}4_132]$ : $\langle T_4,t_zt_yr_y,t_xt_zr_z,t_xt_{1/2}r_{xy},r_{xyz}\rangle$
\item $[\text{I}432]$ : $\langle T_{1/2},r_y,r_z,r_{xy},r_{xyz}\rangle$
\item $[\text{P}4_232]$ : $\langle T_1,r_y,r_z,t_{1/2}r_{xy},r_{xyz}\rangle$
\item $[\text{P}622]$ : $\langle T^\omega_1,r_y,r_z,r_\omega\rangle$
\end{itemize}
\end{lemma}

To list the finite index normal translation subgroups of the above groups, we need the following two lemmas.

\begin{lemma}\label{lem:conjugate}
For a translation $t=(a,b,c)$ of $\mathbb{R}^3$, its conjugates under $r_y$, $r_z$, $r_{xy}$, $r_{xyz}$, $r_\omega$ are the
following translations:
\begin{align*}
r_y^{-1}tr_y&=(-a,b,-c),\quad r_z^{-1}tr_z=(-a,-b,c),\\
r_{xy}^{-1}tr_{xy}&=(b,a,-c),\quad r_{xyz}^{-1}tr_{xyz}=(b,c,a),\\
r_\omega^{-1}tr_\omega&=(-\frac{1}{2}a+\frac{\sqrt{3}}{2}b,-\frac{\sqrt{3}}{2}a-\frac{1}{2}b,c).
\end{align*}
\end{lemma}

\begin{lemma}\label{lem:invariant}
Let $T$ be a discrete group consisting of translations of $\mathbb{R}^3$.

(1) If $T$ is invariant under the conjugation of $r_y,r_z,r_{xyz}$, then there is $u\geq 0$ such that $T$ is one of the three groups:
\[\langle ut_x,ut_y,ut_z\rangle,\quad \langle 2ut_x,ut_y+ut_x,ut_z+ut_x\rangle,\quad \langle 2ut_x,2ut_y,2ut_{1/2}\rangle.\]

(2) If $T$ is invariant under the conjugation of $r_y,r_z,r_\omega$, then there are $u,v\geq 0$ such that $T$ is one of the two groups:
\[\langle ut_\omega,ut_x,vt_z\rangle,\quad \langle 2ut_\omega+ut_x,ut_\omega+2ut_x,vt_z\rangle.\]
\end{lemma}

\begin{proof}
We can assume that $T$ is nontrivial.

(1) Since $T$ is discrete, there is an element $(a,b,c)$ of $T$ having nonzero minimum distance to $(0,0,0)$. Since $(-a,-b,-c)$ is also an element of $T$, by Lemma \ref{lem:conjugate} we can assume that $a,b,c\geq 0$. Since $(-a,-b,c)$ and $(b,c,a)$ are elements of $T$,
\[(-b,-c,2c-a)=(a,b,c)+(-a,-b,c)-(b,c,a)\]
is a nonzero element of $T$. By the choice of $(a,b,c)$, we have
\[b^2+c^2+(2c-a)^2\geq a^2+b^2+c^2.\]
Hence $c(c-a)\geq 0$. We can also have $a(a-c)\geq 0$ and other similar inequalities about $a,b$ and $b,c$. Hence the nonzero ones in $a,b,c$ are equal. Let it be $u$.

If there are two zeros in $a,b,c$, then $T$ contains $\langle ut_x,ut_y,ut_z\rangle$ as a subgroup. For any $t\in T$, there is $t'\in \langle ut_x,ut_y,ut_z\rangle$ such that $t-t'\in [-u/2,u/2]^3$. Then
\[|t-t'|^2\leq \frac{u^2}{4}+\frac{u^2}{4}+\frac{u^2}{4}<u^2.\]
By the choice of $(a,b,c)$, we have $t=t'$ and $T$ equals $\langle ut_x,ut_y,ut_z\rangle$.

If there is exactly one zero in $a,b,c$, then $T$ contains $\langle 2ut_x,ut_y+ut_x,ut_z+ut_x\rangle$ as a subgroup. For any $t\in T$, there is $t'\in \langle 2ut_x,ut_y+ut_x,ut_z+ut_x\rangle$ such that $t-t'\in [-u,u]\times [-u/2,u/2]^2$. Then
\[|t-t'|^2\leq u^2+\frac{u^2}{4}+\frac{u^2}{4}<2u^2.\]
By the choice of $(a,b,c)$, we have $t=t'$ and $T$ equals $\langle 2ut_x,ut_y+ut_x,ut_z+ut_x\rangle$.

Otherwise, $T$ contains $\langle 2ut_x,2ut_y,2ut_{1/2}\rangle$ as a subgroup. For any $t\in T$, there is $t'\in \langle 2ut_x,2ut_y,2ut_{1/2}\rangle$ such that $t-t'\in [-u,u]^2\times [-u/2,u/2]$. Then
\[|t-t'|^2\leq u^2+u^2+\frac{u^2}{4}<3u^2.\]
By the choice of $(a,b,c)$, we have $t=t'$ and $T$ equals $\langle 2ut_x,2ut_y,2ut_{1/2}\rangle$.

(2) By Lemma \ref{lem:conjugate}, for any element $t=(a,b,c)$ in $T$ the two elements
\[(0,0,2c)=(a,b,c)+(-a,-b,c),\quad (0,0,3c)=t+r_\omega^{-1}tr_\omega+r_\omega^{-2}tr_\omega^2\]
belong to $T$. Hence $(0,0,c)$ and $(a,b,0)$ belong to $T$. Consider the subgroups
\[U=\{(a,b,c)\in T\mid c=0\},\quad V=\{(a,b,c)\in T\mid a=b=0\}.\]
Then $T$ is the direct sum of $U$ and $V$. Clearly there is $v\geq 0$ such that $V=\langle vt_z\rangle$.

We can assume that $U$ is nontrivial. Then there is $(a,b,0)$ in $U$ having nonzero minimum distance to $(0,0,0)$. By Lemma \ref{lem:conjugate} we can assume that $a,b>0$. Since
\[(-\frac{1}{2}a+\frac{\sqrt{3}}{2}b,-\frac{\sqrt{3}}{2}a+\frac{3}{2}b,0)= (a,b,0)+(-a,b,0)+(-\frac{1}{2}a+\frac{\sqrt{3}}{2}b,-\frac{\sqrt{3}}{2}a-\frac{1}{2}b,0)\]
is an element of $T$, if $a\neq \sqrt{3}b$, then by the choice of $(a,b,0)$ we have
\[(-\frac{1}{2}a+\frac{\sqrt{3}}{2}b)^2+(-\frac{\sqrt{3}}{2}a+\frac{3}{2}b)^2\geq a^2+b^2.\]
Hence $b(b-\sqrt{3}a)\geq 0$. Similarly we can have $a(\sqrt{3}a-b)\geq 0$. Hence $b=\sqrt{3}a$.

If $b=\sqrt{3}a$, let $u=2a$. Then $U$ contains $\langle ut_\omega,ut_x\rangle$ as a subgroup. By the choice of $(a,b,0)$, it is easy to see that $U$ equals $\langle ut_\omega,ut_x\rangle$. Hence $T$ is $\langle ut_\omega,ut_x,vt_z\rangle$.

If $a=\sqrt{3}b$, let $u=2a/3$. Then $T$ is $\langle 2ut_\omega+ut_x,ut_\omega+2ut_x,vt_z\rangle$.
\end{proof}

\begin{proposition}\label{pro:subgroup}
All finite index normal translation subgroups of the fundamental groups of the 3-orbifolds in Proposition \ref{pro:pairs} are as below, where $m,n\in \mathbb{Z}_+$.
\begin{itemize}
\item $[\text{P}432]$ : $T_{n^3}$, $T_{2n^3}$, $T_{4n^3}$ (with indices $24n^3$, $48n^3$, $96n^3$).
\item $[\text{F}4_132]$ : $T_{2n^3}$, $T_{8n^3}$, $T_{32n^3}$ (with indices $24n^3$, $96n^3$, $384n^3$).
\item $[\text{I}4_132]$ : $T_{4n^3}$, $T_{8n^3}$, $T_{16n^3}$ (with indices $24n^3$, $48n^3$, $96n^3$).
\item $[\text{I}432]$ : $T_{n^3/2}$, $T_{n^3}$, $T_{2n^3}$ (with indices $24n^3$, $48n^3$, $96n^3$).
\item $[\text{P}4_232]$ : $T_{n^3}$, $T_{2n^3}$, $T_{4n^3}$ (with indices $24n^3$, $48n^3$, $96n^3$).
\item $[\text{P}622]$ : $\langle nt_\omega,nt_x,mt_z\rangle$, $\langle 2nt_\omega+nt_x,nt_\omega+2nt_x,mt_z\rangle$ (with indices $12mn^2$, $36mn^2$).
\end{itemize}
\end{proposition}

\begin{proof}
By Lemma \ref{lem:conjugate}, all the listed groups are normal translation subgroups with finite index. Note that $T_1,T_2,T_4,T_{1/2},T_1,T^\omega_1$ in the representations in Lemma \ref{lem:rep.} are the maximal translation subgroups respectively, because for such a group $T$ in the corresponding space group $H$ the $H/T$-action on $\mathbb{R}^3/T$ contains no translations. Then the proof can be finished by using Lemma \ref{lem:invariant}, because the generators of the required group can be uniquely presented by the generators of the maximal translation subgroup and the parameters $u,v$ must have certain forms.

As an example, let $T$ be a finite index normal translation subgroup of the space group $[\text{F}4_132]$. By Lemma \ref{lem:invariant}, there is $u>0$ such that $T$ is one of
\[\langle ut_x,ut_y,ut_z\rangle,\quad \langle 2ut_x,ut_y+ut_x,ut_z+ut_x\rangle,\quad \langle 2ut_x,2ut_y,2ut_{1/2}\rangle.\]

If it is $\langle 2ut_x,ut_y+ut_x,ut_z+ut_x\rangle$, then since $2ut_x=u(2t_x)$, we have $u\in \mathbb{Z}_+$. Hence $u=n$ and $T=T_{2n^3}$.

If it is $\langle ut_x,ut_y,ut_z\rangle$, then since
\[ut_y=-\frac{u}{2}(2t_x)+u(t_y+t_x),\]
we have $u\in \mathbb{Z}_+$ and $2\mid u$. Hence $u=2n$ and $T=T_{(2n)^3}=T_{8n^3}$.

If it is $\langle 2ut_x,2ut_y,2ut_{1/2}\rangle$, then since
\[2ut_{1/2}=-\frac{u}{2}(2t_x)+u(t_y+t_x)+u(t_z+t_x),\]
we have $u\in \mathbb{Z}_+$ and $2\mid u$. Hence $u=2n$ and $T=T_{(4n)^3/2}=T_{32n^3}$.
\end{proof}

\begin{theorem}\label{thm:main}
Up to conjugation, all $G$-actions of maximal possible order $12(g-1)$ on a pair $(\mathbb{T}^3,\Sigma_g)$, with $(\mathbb{T}^3/G,\Sigma_g/G)=(\mathcal{O},\mathcal{F})$, are obtained as the regular coverings of the orbifolds $\mathcal{O}$ corresponding to the
following normal translation subgroups of $\pi_1(\mathcal{O})$ (where $\mathcal{F}$ is the boundary of a regular neighborhood of one the nine singular edges denoted by $\alpha$, $\beta$ or $\gamma$ in Figure \ref{fig:pairs}).
\begin{itemize}
\item $([\text{P}432],\alpha)$ : $T_{n^3}$, $T_{2n^3}$, $T_{4n^3}$.
\item $([\text{F}4_132],\alpha)$ : $T_{2n^3}$, $T_{8n^3}$, $T_{32n^3}$.
\item $([\text{I}4_132],\alpha)$ : $T_{4n^3}$, $T_{8n^3}$, $T_{16n^3}$.
\item $([\text{I}432],\beta)$ : $T_{n^3/2}(2\nmid n)$.
\item $([\text{P}4_232],\beta)$ : $T_{n^3}(2\nmid n)$, $T_{4n^3}(2\nmid n)$.
\item $([\text{P}4_232],\gamma)$ : $T_{n^3}(2\nmid n)$, $T_{2n^3}(2\nmid n)$.
\item $([\text{I}432],\gamma)$ : $T_{n^3/2}(2\nmid n)$.
\item $([\text{I}4_132],\beta)$ : $T_{4n^3}(3\nmid n)$, $T_{8n^3}(3\nmid n)$, $T_{16n^3}(3\nmid n)$.
\item $([\text{P}622],\beta)$ : $T^\omega_{n^2}$, $T^\omega_{3n^2}$.
\end{itemize}
\end{theorem}

\begin{proof}
Note that for each $\mathcal{O}$ in Proposition \ref{pro:pairs} the minimal covering $p_0$ corresponds to the maximal translation subgroup $T_0$ of $\pi_1(\mathcal{O})$, and a regular covering $p$ as in Step \ref{ste:find covering} corresponds to a finite index normal translation subgroup $T$ of $\pi_1(\mathcal{O})$.

Let $\eta$ denote a marked singular edge in $\mathcal{O}$. To finish the proof, we need to check for each of the normal translation subgroups $T$ in Proposition \ref{pro:subgroup} whether $p^{-1}(\eta)$ is connected. In the unknotted cases, i.e. for each of the three edges $\alpha$ in Figure \ref{fig:pairs}, this follows immediately from Lemma \ref{lem:connected} (since the fundamental group of a regular neighborhood of an edge $\alpha$ clearly surjects onto the fundamental group of $\mathcal{O}$, or onto the fundamental groups of the two handlebody orbifolds into which $\mathcal{O}$ splits).

The general case is a consequence of the following two claims.

\begin{claim}\label{cla:connected}
$p_0^{-1}(\eta)$ is a connected graph in $\mathbb{R}^3/T_0$.
\end{claim}

\begin{claim}\label{cla:image}
Let $\hat{i}$ denote the embedding of $p_0^{-1}(\eta)$ in $\mathbb{R}^3/T_0$, then in each case $T_0$ and the image of $\pi_1(p_0^{-1}(\eta))$ in $T_0$ are given in the list below.
\begin{itemize}
\item $([\text{P}432],\alpha)$ : $T_0=T_1$, $\hat{i}_*(\pi_1(p_0^{-1}(\alpha)))=T_1$.
\item $([\text{F}4_132],\alpha)$ : $T_0=T_2$, $\hat{i}_*(\pi_1(p_0^{-1}(\alpha)))=T_2$.
\item $([\text{I}4_132],\alpha)$ : $T_0=T_4$, $\hat{i}_*(\pi_1(p_0^{-1}(\alpha)))=T_4$.
\item $([\text{I}432],\beta)$ : $T_0=T_{1/2}$, $\hat{i}_*(\pi_1(p_0^{-1}(\beta)))=T_1$.
\item $([\text{P}4_232],\beta)$ : $T_0=T_1$, $\hat{i}_*(\pi_1(p_0^{-1}(\beta)))=T_2$.
\item $([\text{P}4_232],\gamma)$ : $T_0=T_1$, $\hat{i}_*(\pi_1(p_0^{-1}(\gamma)))=T_4$.
\item $([\text{I}432],\gamma)$ : $T_0=T_{1/2}$, $\hat{i}_*(\pi_1(p_0^{-1}(\gamma)))=T_4$.
\item $([\text{I}4_132],\beta)$ : $T_0=T_4$, $\hat{i}_*(\pi_1(p_0^{-1}(\beta)))=T_{108}$.
\item $([\text{P}622],\beta)$ : $T_0=T^\omega_1$, $\hat{i}_*(\pi_1(p_0^{-1}(\beta)))=\langle t_\omega,t_x\rangle$.
\end{itemize}
\end{claim}

Now let $T$ be a subgroup in Proposition \ref{pro:subgroup} and $p$ be its corresponding covering. Then there is a covering $q$ such that $p=p_0\circ q$. Since $p_0^{-1}(\eta)$ is connected, by Lemma \ref{lem:connected} the graph $p^{-1}(\eta)=q^{-1}(p_0^{-1}(\eta))$ is connected if and only if
\[\hat{i}_*(\pi_1(p_0^{-1}(\eta)))\cdot T=T_0.\]
Hence assuming the two claims one can check this condition case by case to obtain Theorem \ref{thm:main}.

The two claims can be shown as following.

Since the representation of $\pi_1(\mathcal{O})$ as a space group is given in Lemma \ref{lem:rep.}, one can get the pre-fundamental domain of the $\pi_1(\mathcal{O})$-action on $\mathbb{R}^3$, which consists of points $(x,y,z)$ in $\mathbb{R}^3$ satisfying
\[|(x,y,z)-(0,0,0)|\leq |(x,y,z)-h\cdot(0,0,0)|,\quad \forall\, h\in \pi_1(\mathcal{O}).\]
Modular the action of the stable subgroup of $(0,0,0)$ one can get the fundamental domain of the $\pi_1(\mathcal{O})$-action, and folding up the fundamental domain the 3-orbifold can be obtained. Then the position of the singular edge can be determined and the part of $p_0^{-1}(\eta)$ in a fundamental domain of the $T_0$-action on $\mathbb{R}^3$ can be obtained. Finally, the two claims can be checked.

In section \ref{sec:example}, we will give an explicit example to illustrate this procedure.
\end{proof}

\begin{lemma}[\cite{WWZZ2}]\label{lem:connected}
Suppose that a finite group $G$ acts on $(M,F)$, where $M$ is a 3-manifold with an embedding $i:F\hookrightarrow M$ of a surface. We have diagrams:
\[\xymatrix{
  F \ar[d]_{p} \ar[r]^{i} & M \ar[d]^{p}
  & \pi_1(F) \ar[d]_{p_*} \ar[r]^{i_*} & \pi_1(M) \ar[d]^{p_*} \\
  F/G  \ar[r]^{\hat{i}} & M/G
  & \pi_1(F/G)  \ar[r]^{\hat{i}_*} & \pi_1(M/G)
  }\]
Suppose that $F/G$ is connected. Then $F$ is connected if and only if
\[\hat{i}_*(\pi_1(F/G))\cdot p_*(\pi_1(M))=\pi_1(M/G).\]
\end{lemma}

\begin{remark}
The nine classes in Theorem \ref{thm:main} correspond to the nine examples in \cite{BRWW}. Note that each of the examples must correspond to some $\alpha$ or $\beta$ or $\gamma$. In $[\text{P}4_232]$ and $[\text{I}432]$ the graphs $p_0^{-1}(\beta)$ and $p_0^{-1}(\gamma)$ can be distinguished by the local stable subgroups. In $[\text{I}4_132]$ the graphs $p_0^{-1}(\alpha)$ and $p_0^{-1}(\beta)$ can be distinguished by the property of whether the corresponding surface is knotted or not. Then we have the following correspondence where $\Gamma^1$, $\Gamma^2$, $\Gamma^4$, $\Gamma^{1,2}$, $\Gamma^{2,2}$, $\Gamma^{4,4}$, $\Gamma^{4,8}$, $\Gamma'^4$, $\Gamma^P$ are the nine graphs defined in the examples in \cite{BRWW}.
\begin{itemize}
\item $[\text{P}432]$ : $p_0^{-1}(\alpha)=\Gamma^1/T_1$.
\item $[\text{F}4_132]$ : $p_0^{-1}(\alpha)=\Gamma^2/T_2$.
\item $[\text{I}4_132]$ : $p_0^{-1}(\alpha)=\Gamma^4/T_4$, $p_0^{-1}(\beta)=\Gamma'^4/T_4$.
\item $[\text{I}432]$ : $p_0^{-1}(\beta)=\Gamma^{1,2}/T_{1/2}$, $p_0^{-1}(\gamma)=\Gamma^{4,8}/T_{1/2}$.
\item $[\text{P}4_232]$ : $p_0^{-1}(\beta)=\Gamma^{2,2}/T_1$, $p_0^{-1}(\gamma)=\Gamma^{4,4}/T_1$.
\item $[\text{P}622]$ : $p_0^{-1}(\beta)=\Gamma^P/T^\omega_1$.
\end{itemize}
Hence the Claim \ref{cla:connected} can be derived from the nine examples. The connectedness of $p^{-1}(\eta)$ in the proof of Theorem \ref{thm:main} can also be checked directly by using the above correspondence, and this gives an alternative approach to the proof of Theorem \ref{thm:main}, on the basis of the graphs in \cite{BRWW}.
\end{remark}

\begin{proof}[Proof of Theorem \ref{thm:classification}]
Since the order of $G$ is $12(g-1)$, one can compute the genus $g$ in Theorem \ref{thm:main} case by case. Then the list can be obtained.
\end{proof}

\begin{remark}
One can also consider similar questions for graphs and handlebodies, where we need to replace the genus of the surface by the algebraic genus of the graph or handlebody, which equals the rank of the fundamental group. Then for extendable actions over $\mathbb{T}^3$, the upper bound is $12(g-1)$ and the classification result is parallel to the case of surfaces.
\end{remark}

\section{An example}\label{sec:example}
Here we give an example to illustrate the correspondence between the space group and the Euclidean orbifold which was first given by Dunbar. Via this we can obtain pre-images of the marked singular edges in Proposition \ref{pro:pairs}. Then we can check the claims in the proof of Theorem \ref{thm:main}.

Consider the space group $[\text{I}4_132]$ which has the representation
\[\langle t_xt_x,t_yt_y,t_xt_yt_z,t_zt_yr_y,t_xt_zr_z,t_xt_{1/2}r_{xy},r_{xyz}\rangle.\]
Note that the subgroup generated by the first three generators is $T_4$, which is the maximal translation subgroup of $[\text{I}4_132]$.

Let $S$ be the orbit of $(0,0,0)$ under the action of $T_4$. It is invariant under $r_y$, $r_z$, $r_{xy}$ and $r_{xyz}$. Hence the orbit of $(0,0,0)$ under the action of $[\text{I}4_132]$ contains
\[S,\quad t_zt_yS,\quad t_xt_zS,\quad t_xt_{1/2}S.\]
Then it is easy to see that this orbit is the same as the orbit of $(0,0,0)$ under the action of $T_{1/2}$. Hence the closest points to $(0,0,0)$ in the orbit are
\[(\pm1,0,0),\quad (0,\pm1,0),\quad (0,0,\pm1),\quad (\pm\frac{1}{2},\pm\frac{1}{2},\pm\frac{1}{2}).\]
Then we can get a pre-fundamental domain of the space group $[\text{I}4_132]$ as shown in Figure \ref{fig:prefund}, which is a truncated octahedron in the cube $[-1/2,1/2]^3$. In Figure \ref{fig:prefund}, several rotation axes are also pictured. We write down some of them.
\begin{itemize}
\item $\pi$-rotation around $x=1/2,z=0$ : $t_xt_x(t_xt_yt_z)^{-1}t_zt_yr_y$.
\item $\pi$-rotation around $x=0,y=-1/2$ : $(t_xt_yt_z)^{-1}t_xt_zr_z$.
\item $\pi$-rotation around $z=-1/4,x-y=1/2$ : $(t_xt_xt_yt_y)^{-1}t_xt_{1/2}r_{xy}(t_xt_yt_z)$.
\item $\pi$-rotation around $x=1/4,z-y=1/2$ : $r_{xyz}(t_xt_x)^{-1}t_xt_{1/2}r_{xy}r_{xyz}^{-1}$.
\item $\pi$-rotation around $z=1/4,x+y=1/2$ : $t_xt_{1/2}r_{xy}(t_xt_yt_z)^{-1}t_xt_zr_z$.
\end{itemize}
Since $x=y=z$ is a rotation axe of order $3$, we can get other axes in the figure. Note that the quotient group $[\text{I}4_132]/T_4$ has order $24$ and the fundamental domain of $T_4$ has volume $4$, hence the fundamental domain of $[\text{I}4_132]$ has volume $1/6$. Since each hexagon in Figure \ref{fig:prefund} cuts the smaller cube into pieces having equal volume, the volume of the pre-fundamental domain is $1/2$. Hence we know that the stable subgroup of $(0,0,0)$ is generated by $r_{xyz}$.

\begin{figure}[h]
\includegraphics{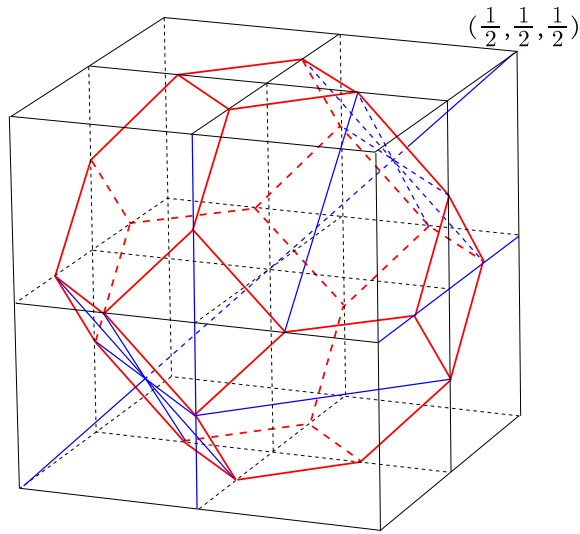}
\caption{Red polyhedron in $[-1/2,1/2]^3$ is a pre-fundamental domain of $[\text{I}4_132]$. Blue lines indicate part of the rotation axes.}\label{fig:prefund}
\end{figure}

\begin{figure}[h]
\includegraphics{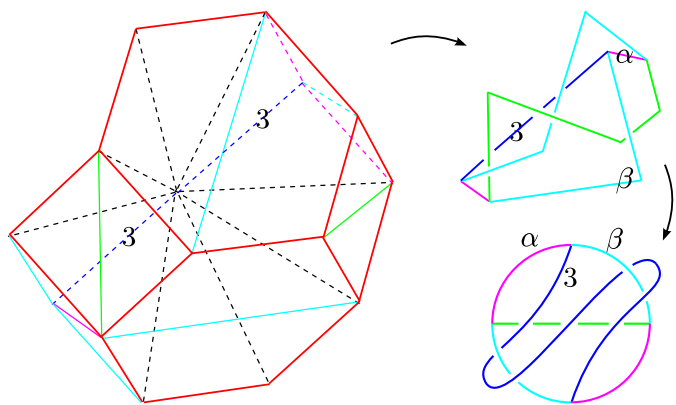}
\caption{The left picture is the fundamental domain of $[\text{I}4_132]$. Folding it up we get the right upper picture. It is isotopic to the right lower picture which is given in the Dunbar's list.}\label{fig:orbifold}
\end{figure}

Then we can obtain the fundamental domain of the space group $[\text{I}4_132]$ as in the left picture of Figure \ref{fig:orbifold}. Here we use different colours to distinguish the rotation axes. To get the orbifold, we first fold up the squares and (part of) the hexagons; then we fold up the remaining faces around the axe of order $3$. The orbifold has underlying space $\mathbb{S}^3$ and singular set as shown in the right pictures of Figure \ref{fig:orbifold}.

The corresponding orbifolds of other space groups can be obtained in a similar way. At what follows, we will check the two claims in the case of $([\text{I}4_132],\beta)$.

From Figure \ref{fig:prefund} and Figure \ref{fig:orbifold} we see that the singular edge $\beta$ corresponds to the union of the following two edges
\[\{(\frac{1}{4},\frac{1}{4},\frac{1}{4}),(\frac{1}{2},0,\frac{1}{4})\},\quad \{(\frac{1}{2},0,-\frac{1}{4}),(0,-\frac{1}{2},-\frac{1}{4})\}.\]
Via the $\pi$-rotation around $z=1/4,x+y=1/2$ these become the two edges
\[\{(\frac{1}{4},\frac{1}{4},\frac{1}{4}),(\frac{1}{2},0,\frac{1}{4})\},\quad \{(\frac{1}{2},0,\frac{3}{4}),(1,\frac{1}{2},\frac{3}{4})\},\]
in $[0,1]^3$. To get the part of $p_0^{-1}(\beta)$ in $[0,2]^2\times[0,1]$ which is a fundamental domain of $T_4$, we just need to apply the following actions to the two edges:
\[r_{xyz},\quad t_xt_xt_zt_yr_y,\quad t_xt_xt_yt_y(t_xt_yt_z)^{-1}t_xt_zr_z,\quad (t_xt_yt_z)t_zt_yr_yt_xt_zr_z.\]
Figure \ref{fig:graph} shows the graph, which is the one given in \cite{BRWW}. Since the $T_4$-action folds up the opposite faces, the graph is connected in $\mathbb{R}^3/T_4$.

\begin{figure}[h]
\includegraphics{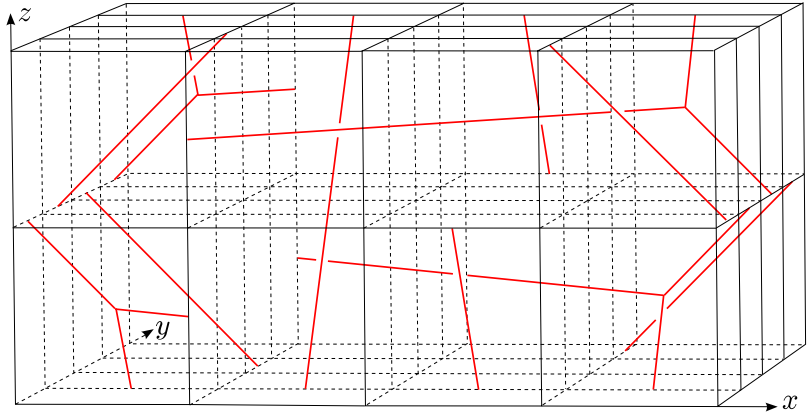}
\caption{Red lines indicate the graph $p_0^{-1}(\beta)$ in $[0,2]^2\times[0,1]$.}\label{fig:graph}
\end{figure}

Abstractly this graph is the complete graph with $4$ vertices. It has fundamental group the free group of rank $3$. To compute the image of $\pi_1(p_0^{-1}(\beta))$ in $T_4$, we can choose a base point of the graph in the interior of $[0,2]^2\times[0,1]$ and find the image of each generator of $\pi_1(p_0^{-1}(\beta))$. Note that once an edge passes through a face of the fundamental domain of $T_4$ it corresponds to one of the six elements
\[t_xt_x,\quad t_yt_y,\quad t_xt_yt_z,\quad (t_xt_x)^{-1},\quad (t_yt_y)^{-1},\quad (t_xt_yt_z)^{-1}.\]
As a result, we have
\[\hat{i}_*(\pi_1(p_0^{-1}(\beta)))=\langle 6t_x,6t_y,6t_{1/2}\rangle=T_{6^3/2}=T_{108}.\]
This finishes the proof of the two claims in the case of $([\text{I}4_132],\beta)$.

By Proposition \ref{pro:subgroup}, the possible finite index normal translation subgroups $T$ of $[\text{I}4_132]$ is $T_{4n^3}$ or $T_{8n^3}$ or $T_{16n^3}$.

If $3\mid n$, then $T$ is a subgroup of $T_{108}$ and $T_{108}\cdot T=T_{108}\neq T_4$.

If $3\nmid n$, then $T_{108}\cdot T_{4n^3}=T_4$ for any $n\in \mathbb{Z}_+$. Hence $T_{108}\cdot T_{8n^3}=T_{108}\cdot T_{16n^3}=T_4$ for any $n\in \mathbb{Z}_+$ satisfying $3\nmid n$.

\bibliographystyle{amsalpha}

\begin{thebibliography}{WWZZZ}
\bibitem[BMP]{BMP}
M. Boileau, S. Maillot, J. Porti, {\it Three-dimensional orbifolds and their geometric structures}, Panoramas et Synth\`eses 15. Soci\'et\'e Math\'matique de France, Paris 2003.

\bibitem[BRWW]{BRWW}
S. Bai, V. Robins, C. Wang, S. C. Wang, {\it The maximally symmetric surfaces in the 3-torus}, Geom. Dedicata 189 (2017), 79-95.

\bibitem[BWW]{BWW}
S. Bai, C. Wang, S. C. Wang, {\it Minimal surfaces in the three dimensional sphere with high symmetry}, Preprint.

\bibitem[Du]{Du}
W. D. Dunbar, {\it  Geometric orbifolds}, Rev. Mat. Univ. Complut. Madrid 1 (1988), no.1-3, 67-99.

\bibitem[Ha]{Ha}
T. Hahn (Ed.), {\it International tables for crystallography}, D. Reidel Publishing Company, (2005).

\bibitem[Hu]{Hu}
A. Hurwitz, {\it \"Uber algebraische Gebilde mit eindeutigen Transformationen in sich}, Math. Ann. 41 (1892), no.3, 403-442.

\bibitem[KPS]{KPS}
H. Karcher, U. Pinkall, I. Sterling, {\it New minimal surfaces in $S^3$}, J. Differential Geom. 28 (1988), 169-185.

\bibitem[La]{La}
H. B. Lawson, Jr., {\it Complete minimal surfaces in $S^3$}, Ann. of Math. (2) 92 (1970), 335-374.

\bibitem[MY]{MY}
W. H. Meeks, S. T. Yau, {\it The equivariant Dehn's lemma and loop theorem}, Comment. Math. Helv. 56 (1981), no.2, 225-239.

\bibitem[Po]{Po}
K. Polthier, {\it New minimal surfaces in $H^3$}, Theoretical and numerical aspects of geometric variational problems, Proc. Workshop, Canberra/Australia, Proc. Cent. Math. Anal. Aust. Natl. Univ. 26 (1991), 201-210.

\bibitem[SW]{SW}
H. Seifert, C. Weber, {\it Die beiden Dodekaederr\"aume}, Math. Z. 37 (1933), 237-253.

\bibitem[Th1]{Th}
W. P. Thurston, {\it The geometry and topology of three-manifolds}, Lecture Notes Princeton University 1978.

\bibitem[Th2]{Th2}
W. P. Thurston, {\it Three-Dimensional Geometry and Topology (Edited by S. Levy)}, Princeton University Press, Princeton, New Jersey 1997.

\bibitem[WWZZ1]{WWZZ1}
C. Wang, S. C. Wang, Y. M. Zhang, B. Zimmermann, {\it Extending finite group actions on surfaces over $S^3$}, Topology Appl. 160 (2013), no.16, 2088-2103.

\bibitem[WWZZ2]{WWZZ2}
C. Wang, S. C. Wang, Y. M. Zhang, B. Zimmermann, {\it Embedding surfaces into $S^3$ with maximum symmetry}, Groups Geom. Dyn. 9 (2015), no.4, 1001-1045.

\bibitem[WWZZ3]{WWZZ3}
C. Wang, S. C. Wang, Y. M. Zhang, B. Zimmermann, {\it Embedding compact surfaces into the 3-dimensional Euclidean space with maximum symmetry}, Sci. China Math. 60 (2017), no.9, 1599-1614.

\bibitem[Zi1]{Z1}
B. Zimmermann, {\it \"Uber Abbildungsklassen von Henkelk\"orpern}, Arch. Math. 33 (1979), 379-382.

\bibitem[Zi2]{Z2}
B. Zimmermann, {\it On the Hantzsche-Wendt manifold}, Monatsh. Math. 110 (1990), no.3-4, 321-327.
\end{thebibliography}

\end{document}